\documentclass[12pt]{article}
\usepackage{amsfonts}
\usepackage{amssymb}
\usepackage{amsmath}
\usepackage{color}
\usepackage{bm}
\usepackage{amscd} 
\usepackage[all]{xy}
\usepackage{hyperref}
\usepackage{amsthm}
\allowdisplaybreaks
\newtheorem{thm}{Theorem}[section]
\newtheorem{cor}[thm]{Corollary}
\newtheorem{lem}[thm]{Lemma}

\newtheorem{defi}[thm]{Definition}

\newtheorem{propp}[thm]{Proposition}
\newtheorem{remark}[thm]{Remark}

\title{\bf{On generalized cyclotomic derivations}}

\smallskip 
\author{Sakshi Gupta \footnote{Dept. of Mathematics, Indian Institute of Technology, Delhi,  India. Email: maz188315@iitd.ac.in, Orchid id: 0000-0003-1012-0420} $\qquad$ Surjeet Kour \footnote{Dept. of Mathematics, Indian Institute of Technology, 
Delhi, India. Email: koursurjeet@gmail.com }}
\date{}

\begin{document}

\maketitle

\medskip
\begin{abstract}
In this article we study the field of rational constants and Darboux polynomials of a generalized cyclotomic $K$-derivation $d$ of $K[X]$. It is shown that $d$ is  without Darboux polynomials if and only if $K(X)^d=K$. Result is also studied in the tensor product of polynomial algebras.
\end{abstract}

\section{Introduction}
Throughout this article $K$ denotes a field of characteristic zero,  $K[X]=K[x_1,x_2,\ldots, x_n]$  is the  polynomial algebra in $n$ variables over $K$ and $K(X)$ denotes the field of fractions of $K[X]$. Let  $d$ be a $K$-derivation of $K[X]$ and  $K[X]^d$  denote the algebra of constants of $d$. The $K$-derivation $d$ of $K[X]$ uniquely extends to a $K$-derivation of $K(X)$ and we continue to denote it by $d$. $K(X)^d$ represents the field of rational constants of $d$, that is, $K(X)^d=\lbrace f \in K(X); d(f)=0\rbrace$ and $K[X]^d \subseteq K(X)^d$.
\par
A non-constant polynomial $ f \in K[X]$ is said to be a Darboux polynomial of $d$ if $d(f)=\lambda f$, for some $\lambda \in K[X]$ and in this case $\lambda$ is called the  co-factor of $f$. We say $d$ is without Darboux polynomials if $d$ has no darboux polynomials. It is easy to observe that if $d$ is without Darboux polynomials then $K(X)^d =K$ but the converse of the above statement is not true, in general. One can refer to $(\cite{NA94},\cite{NA96})$ for counter examples. In this paper we study a class of  monomial derivations for which $K(X)^d=K$ if and only if $d$ is without Darboux polynomials. Note that, a derivation $d$ of $K[X]$ is said to be a monomial derivation if $d(x_i)$ is a monomial for every $ 1 \leq i \leq n$.
\par In $\cite{OM11}$,  Nowicki and Ollagnier studied  the Darboux polynomials  and field of rational constants of a monomial derivation over the polynomial algebra $K[X]$.  For $1 \leq i \leq n$, let $\alpha_i =(\alpha_{i1}, \ldots, \alpha_{i n})\in \mathbb{N}^n$. Consider the  monomial derivation $d$ given by $d(x_i)=X^{\alpha_i}$, where $X^{\alpha_i}$ denotes the monomial $x_1^{\alpha_{i 1}} \cdots x_n^{\alpha_{i n}}$. Then one can associate a matrix $A$ given by  $A =[\alpha_{ij}]-I$ with the monomial derivation  d. Let $w_d$ denote the determinant of the matrix A. A monomial derivation $d$ is said to be normal if $w_d\neq 0$  and $\alpha_{i i}=0 ~ \forall~ 1 \leq i \leq n$. In $\cite{OM11}$,  Nowicki and Ollagnier  proved that for a  normal derivation $d$,  $K(X)^d=K$ if and only if $d$ is without Darboux  polynomials. In this article, they also raised the similar question in case of $w_d=0$. For $n=3$, an independent proof was given to show that the result is true even if $w_d=0$. It was also observed that the idea used to prove the result for $n=3$ case can not be extended further. At the end of the article they mentioned the following example of monomial derivation $d$ on $K[x,y,z,w ]$;
$$ d(x)= w^2, ~~ d(y)=zw, ~~ d(z)=y^2, ~~ d(w)=xy$$
and raised the same question about field of constants and Darboux polynomials.
Note that for this derivation $w_d =0$. Our study of field of constants of monomial derivations  is motivated by  aforesaid example.

In this article we study a large class of monomial derivations for which $K(X)^d=K$ if and only if $d$ is without Darboux polynomial. The example mentioned above is a very particular case of our result. Further our result is independent of the condition $w_d\neq 0$.

\section{Result}
Let $s$ be a non-negative integer. A derivation $d$ on $K[X]$ is said to be homogeneous of degree $s$, if 
$$
d(A^{m}) \subseteq d(A^{m+s}) ~~~~~~~  \forall ~ m \geq 1,
$$
where $A^{m}$ is the $K$-subspace of all the homogeneous polynomials of degree $m$.  In particular a monomial derivation $d$ on $K[X]$ is  homogeneous  of degree $s$ if for each $1 \leq i \leq n$, $d(x_i)'s$ are monomials of  total degree $s+1$.

For a homogeneous $K$-derivation $d$ of $K[X]$ the following is a well known result  $(\cite{OM93}$ Lemma $2.1)$.
\begin{lem}{\label{deg}}
Let $d$ be a homogeneous $K$-derivation of $K[X]$ of degree $s$. If $f \in K[X]$ is a Darboux polynomial of $d$ with the co-factor  $\lambda$, then $\lambda$ is homogeneous polynomial of degree $s$ and all the homogeneous components of $f$ are also Darboux polynomials with the same co-factor $\lambda$.
\end{lem}


Before proceeding further we fix some notations and give some more definitions. Let $n \geq 2$ be a positive integer. A derivation $d$ of $K[X]$ is called cyclotomic if, for $1\leq i \leq n-1$, $d(x_i)= x_{i+1}$ and $d(x_n)=x_1$. In the same line we have defined the generalized cyclotomic derivation. Let $S=\{x_1, \dots, x_n\}$ denote the set of $n$-variables, $K[S]$ be the $K$-algebra generated by $S$  and $k>1$ be a positive integer.  A monomial  derivation $d$ of $K[S]$ is said to be a generalized cyclotomic derivation if we can split $S$  into $k$ parts $($say $S_i,~1\leq i \leq k)$ such that $d(S_i)\subseteq K[S_{i+1}]$ for $1 \leq i \leq n-1$  and  $d(S_n) \subseteq K[S_1]$, where $K[S_i]$ denotes the $K$-algebra generated by $S_i$.

Let us  redefine the variables as
$S_i= \{ x_{i ,1}, \ldots, x_{i, t_i}\}$ for $1 \leq i \leq k$ and take $S =\cup S_i$. Then,
\begin{defi}
A derivation $d$ of $K[S]$ is said to be generalized cyclotomic  derivation if 
$$d(x_{i,j})=x_{i+1,1}^{\alpha_{i_j,1}} x_{i+1,2}^{\alpha_{i_j,2}} \cdots x_{i+1,t_{i+1}}^{\alpha_{i_j,t_{i+1}}} \hspace{3mm}\forall ~1 \leq j \leq t_i,\hspace{1.5mm}1 \leq i \leq k-1$$    
and 
$$d(x_{k,j})=x_{1,1}^{\alpha_{k_j,1}} x_{1,2}^{\alpha_{k_j,2}} \cdots x_{1,t_{1}}^{\alpha_{k_j,t_{1}} } \hspace{3mm} \forall ~ 1 \leq j \leq t_k,$$
where $\alpha_{i_{j},l}$ for every $1 \leq j \leq t_i$, $1 \leq l \leq t_{i+1}$, $1 \leq i \leq k-1$ and $\alpha_{k_{j},l}$ for every $1 \leq j \leq t_k$, $1 \leq l \leq t_{1}$ are non negative integers.  
\end{defi}
Let $s$ be a positive integer.
 We say $d$ is homogeneous generalized cyclotomic derivation of degree $s-1$ if $d(x_{i j})$ are monomials of  total degree $s$.  Now we state our main result.


\begin{thm}{\label{MT}}
Let $d$ be a homogeneous generalized cyclotomic derivation of $K[S]$ of degree $s-1$ as defined above. Then $d$ is without darboux polynomials if and only if $K(S)^{d}=K$.
\end{thm}

\begin{proof}
$(\Rightarrow)$ Easy to prove.
\\
$(\Leftarrow)$ Assume that $K(S)^d=K$.  Suppose $d$ has a Darboux polynomial $f$ such that $d(f)=\lambda f$. Then by Lemma \ref{deg}, we may assume that $f$ is homogeneous and $\lambda$ is a homogeneous of degree $s-1$. Write $\lambda$ as:
$$
\lambda= \sum \limits_{ \sum \limits_{1 \leq i \leq k} | \beta_{i}|=s-1}  a_{\beta} X_{1}^{\beta_{1}} X_{2}^{\beta_{2}} \cdots X_{k}^{\beta_{k}},
$$
where  $X_i^{\beta_i}$ denotes the monomial $x_{i,1}^{\beta_{i,1}} \cdots x_{i,t_i}^{\beta_{i,t_i}}$ and $|\beta_i| = \sum \limits_{j=1}^{j=t_i}\beta_{i j}$ for $1 \leq i \leq k$.

Let $N = (1 +s + s^2 + \cdots + s^{k-1})$ and let $\xi$ be the primitive $N$-th root of unity. 
For $0 \leq i \leq k-1$, define $q_i=\sum \limits_{j=0}^{i} s^j$. Observe that $q_{k-1}=N$. Consider
 a $K$-automorphism $\sigma$ of $K[S]$ given by: 
\begin{equation*}
\sigma(x_{k-i,j})=\xi^{q_i} x_{k-i,j} \hspace{1mm}  \forall ~~ 1 \leq j \leq t_{k-i}.
\end{equation*}

Moreover, 
\begin{align*}
\sigma^{-1} d \sigma (x_{k-i,j}) &= \sigma^{-1} d({\xi^{q_i} x_{k-i,j}}) \\
&=\xi^{q_i} \sigma^{-1} d (x_{k-i,j}) \\
&=\xi^{q_i} \sigma^{-1}\left(x_{k-i+1,1}^{\alpha_{(k-i)_j,1}} x_{k-i+1,2}^{\alpha_{(k-i)_j,2}} \cdots x_{k-i+1,t_{k-i+1}}^{\alpha_{(k-i)_j,t_{k-i+1}}}\right) \\
&=\xi^{q_i} \xi^{-q_{i-1} \left(\alpha_{(k-i)_j,1}+\cdots + \alpha_{(k-i)_j,t_{k-i+1}}  \right) } d(x_{k-i,j}) \\
&=\xi^{q_i} \xi^{-s(q_{i-1})} d(x_{k-i,j})\\
&=\xi d(x_{k-i,j}).
\end{align*}

Therefore, we have $\sigma^{-1}d\sigma = \xi d.$
Let $F=\prod \limits_{i=0}^{N-1} \sigma^i(f)$. Clearly, $F$ is not a constant polynomial.  Furthermore,

\begin{align*}
d(F)&=d\left(\prod \limits_{i=0}^{N-1} \sigma^i(f)\right) \\
&=\sum \limits_{i=0}^{N-1} \sigma^{0}(f) \cdots d(\sigma^{i}(f)) \cdots \sigma^{N-1}(f)\\
&=\sum \limits_{i=0}^{N-1} \sigma^{0}(f) \cdots \xi^i\sigma^{i}d(f) \cdots \sigma^{N-1}(f) \\
&=\sum \limits_{i=0}^{N-1} \sigma^{0}(f) \cdots \xi^i\sigma^{i}(\lambda f) \cdots \sigma^{N-1}(f) \\
&=\left(\sum \limits_{i=0}^{N-1} \xi^i \sigma^{i}(\lambda) \right) ~\sigma^{0}(f) \cdots  \sigma^{N-1}(f) \\
&=\Lambda F,
\end{align*}
where $\Lambda=\sum \limits_{i=0}^{N-1} \xi^i \sigma^{i}(\lambda)$. \\
Now, let us do the precise calculation for $\Lambda$. If we look at the $m$-th term in the sum, we have 
\begin{align*}
\xi^m \sigma^{m}(\lambda)&=\xi^m \sigma^{m}\left[\sum\limits_{ \sum \limits_{1 \leq i \leq k} | \beta_{i}|=s-1}  a_{\beta} X_{1}^{\beta_{1}} X_{2}^{\beta_{2}} \cdots X_{k}^{\beta_{k}}\right] \\
&=\xi^m \left[\sum\limits_{ \sum \limits_{1 \leq i \leq k} | \beta_{i}|=s-1} a_{\beta} ~\sigma^{m} \left(X_1^{\beta_1}\right)\ldots \sigma^{m}\left(X_k^{\beta_k}\right)\right] \\
&=\xi^m\left[\sum\limits_{ \sum \limits_{1 \leq i \leq k} | \beta_{i}|=s-1} \left(\prod_{l=1}^{l=k}\sigma^{m} \left(x_{l1}^{\beta_{l1}}\right)  \cdots  \sigma^{m}\left(x_{lt_l}^{\beta_{lt_l}}\right)\right)\right]\\
&=\xi^m\left[\sum\limits_{ \sum \limits_{1 \leq i \leq k} | \beta_{i}|=s-1}  a_{\beta} \prod_{l=1}^{l=k}\xi^{mp_lq_{k-l}}X_{l}^{\beta_{l}}\right],
\end{align*}
where $p_l= \beta_{l1}+ \cdots + \beta_{lt_l} $ for all $1 \leq l \leq k$.  Therefore,
\begin{align*}
\xi^m \sigma^{m}(\lambda) &=\xi^m\left[\sum\limits_{ \sum \limits_{1 \leq i \leq k} | \beta_{i}|=s-1}  a_{\beta} \xi^{m\left(\sum \limits_{l=1}^{l=k} p_lq_{k-l}\right)}X_{1}^{\beta_{1}} X_{2}^{\beta_{2}} \cdots X_{k}^{\beta_{k}} \right]\\
&=\xi^m\left[\sum\limits_{ \sum \limits_{1 \leq i \leq k} | \beta_{i}|=s-1}  a_{\beta} \xi^{m\left(\sum \limits_{l=2}^{l=k} p_lq_{k-l}\right)}X_{1}^{\beta_{1}} X_{2}^{\beta_{2}} \cdots X_{k}^{\beta_{k}} \right]\\
&=\sum\limits_{ \sum \limits_{1 \leq i \leq k} | \beta_{i}|=s-1}  a_{\beta} \xi^{m(\delta +1)}X_{1}^{\beta_{1}} X_{2}^{\beta_{2}} \cdots X_{k}^{\beta_{k}}, \\
\end{align*}
where $\delta=\sum \limits_{l=2}^{l=k} p_lq_{k-l}$. More precisely,

\begin{align*}
\delta =&  \sum_{l=2}^{l=k}q_{k-l}p_l\\ 
=& \sum_{l=2}^{l=k} \left(\sum \limits_{j=0}^{k-l} s^{j} \right) \left(\beta_{l,1} + \cdots + \beta_{l,t_l} \right) \\
=& \sum_{l=0}^{l=k-2}s^l(\beta_{2,1} + \cdots + \beta_{2,t_2} + \cdots + \beta_{k,1} + \cdots + \beta_{k-l,t_{k-l}}) \\
\leq& \sum_{l=0}^{l=k-2}s^l (s-1)
=s^{k-1}-1.
\end{align*}
This implies that $0 < 1+\delta \leq s^{k-1} < N$. Therefore, $\xi^{1+\delta} \neq 1$. 
Hence,
\begin{align*}
\Lambda&=\sum \limits_{m=0}^{N-1} \left[\sum\limits_{ \sum \limits_{1 \leq i \leq k} | \beta_{i}|=s-1} \xi^{m(\delta +1)} a_{\beta} X_{1}^{\beta_{1}} X_{2}^{\beta_{2}} \cdots X_{k}^{\beta_{k}}  \right] \\
&=  \sum \limits_{ \sum \limits_{1 \leq i \leq k} | \beta_{i}|=s-1} \left( \sum \limits_{m=0}^{N-1} \xi^{(\delta+1) m} \right) a_{\beta} X_{1}^{\beta_{1}} X_{2}^{\beta_{2}} \cdots X_{k}^{\beta_{k}}   \\
&=  \sum \limits_{ \sum \limits_{1 \leq i \leq k} | \beta_{i}|=s-1} \Bigg( \frac{1-\xi^{N(\delta+1)}}{1-\xi^{\delta+1}} \Bigg) a_{\beta} X_{1}^{\beta_{1}} X_{2}^{\beta_{2}} \cdots X_{k}^{\beta_{k}}   \\
&=0.
\end{align*}
Therefore, $d(F)=\Lambda F=0$. In other words, $F \in K(S)^{d}$, a contradiction.
\end{proof}

\begin{remark}
From Theorem \ref{MT}, we can prove that the monomial derivation $d$ of $K[x,y,z,w]$ defined by:
$$d(x)=w^2, d(y)=zw, d(z)=y^2, d(w)=xy$$ in $\cite{OM11}$ has no Darboux polynomials if and only if $K(x,y,z,w)^{d}=K.$
\end{remark}

\begin{cor}
 The Jouanolou derivation $d$ of $K[x_1,x_2,\ldots,x_n]$ defined by:
$$d(x_1)=x_2^s, d(x_2)=x_3^s, \ldots, d(x_{n-1})=x_{n}^s, d(x_n)=x_1^s,$$ for $s \geq 1$ and $n \geq 2$ has no Darboux polynomials if and only if $K(x_1,x_2,\ldots,x_n)^{d}=K$ . 
\end{cor}

\section{Generalized cyclotomic derivation in Tensor Product}

Let $m$ and $n$ be positive integers. Assume that $K[X]=K[x_1,\ldots,x_n]$ and $K[Y]=K[y_1,\ldots,y_m]$ are polynomial algebras. Then $K[X] \otimes_K K[Y] \cong K[X,Y] = K[x_1, \ldots, x_n,y_1, \ldots,y_m]$ is a polynomial algebra. If $d_1$ and $d_2$ are $K$-derivations of $K[X]$ and $K[Y]$ respectively, then $d=d_1 \otimes 1 + 1 \otimes d_2$, denoted by $d_1 \oplus d_2$, is the $K$-derivation of $K[X,Y]$ such that $d|_{K[X]}=d_1$ and $d|_{K[Y]}=d_2.$

In $\cite{OM04}$ Nowicki and Ollagnier studied the Darboux polynomial of the tensor product of polynomial algebras and have proved the following result:

\begin{lem} $(\cite{OM04}$ Corollary $3.2)${\label{DPS}}
Let $d_1$ and $d_2$ be homogeneous $K$-derivations of $K[X]$ and $K[Y]$ of degree $s \geq 1$. If $d_1$ and $d_2$ are without Darboux polynomials, then $d_1 \oplus d_2$ is also without Darboux polynomials.
\end{lem}


Using Lemma \ref{DPS} and our result on generalized cyclotomic derivations we have the following:

\begin{thm}
Let $d_1$ and $d_2$ be homogeneous generalized cyclotomic derivations of $K[X]$ and $K[Y]$ of degree $s\geq 1$. Then, $d_1 \oplus d_2$ is without Darboux polynomial if and only if $K(X,Y)^{d_1 \oplus d_2}=K$.
\end{thm}
\begin{proof}
$(\Rightarrow)$ Trivial to prove.\\
$(\Leftarrow)$ Let $K(X,Y)^{d_1 \oplus d_2}=K$. As $K(X)^{d_1} \subseteq K(X,Y)^{d_1 \oplus d_2} =K$, we have $K(X)^{d_1} = K$. Similarly, $K(Y)^{d_2}=K$. Therefore, from Theorem $\ref{MT}$, $d_1$ and $d_2$ are without Darboux polynomials. Then by Lemma $\ref{DPS}$, $d_1 \oplus d_2$ is also without darboux polynomials.  
\end{proof}

\section*{Acknowledgements}
This research is partially supported by DST-INSPIRE grant IFA-13/MA-30.  The first author is financially supported by CSIR, India.

\end{document}